\documentclass[12pt,reqno,a4paper]{amsart}
\usepackage{blindtext}

\usepackage[dvipsnames]{xcolor}
\usepackage{mathtools}
\mathtoolsset{showonlyrefs}
\usepackage{etoolbox}
\usepackage{amscd}
\usepackage{blindtext}
\usepackage{latexsym}
\usepackage[pagewise]{lineno}
\usepackage[utf8]{inputenc}
\usepackage{exscale}
\usepackage{amsmath}
\usepackage{amssymb}
\usepackage{amsfonts}
\usepackage{mathrsfs}
\usepackage{amsbsy}
\usepackage{relsize}
\usepackage{amsmath, amsthm}

\usepackage{comment}
\PassOptionsToPackage{reqno}{amsmath}

\definecolor{mypink1}{rgb}{0.858, 0.188, 0.478}
\definecolor{mypink2}{RGB}{219, 48, 122}
\definecolor{mypink3}{cmyk}{0, 0.7808, 0.4429, 0.1412}
\definecolor{mygray}{gray}{0.6}
\definecolor{venetianred}{rgb}{0.78, 0.03, 0.08}
\definecolor{sapphire}{rgb}{0.03, 0.15, 0.4}
\definecolor{utahcrimson}{rgb}{0.83, 0.0, 0.25}
\definecolor{trueblue}{rgb}{0.0, 0.45, 0.81}
\definecolor{carminered}{rgb}{1.0, 0.0, 0.22}
\definecolor{cobalt}{rgb}{0.0, 0.28, 0.67}
\definecolor{cornflowerblue}{rgb}{0.39, 0.58, 0.93}
\definecolor{darkmagenta}{rgb}{0.55, 0.0, 0.55}
\definecolor{electricultramarine}{rgb}{0.25, 0.0, 1.0}
\definecolor{falured}{rgb}{0.5, 0.09, 0.09}
\definecolor{hancornflowerblue}{rgb}{0.32, 0.09, 0.98}
\definecolor{mahogany}{rgb}{0.75, 0.25, 0.0}
\definecolor{oucrimsonred}{rgb}{0.6, 0.0, 0.0}
\definecolor{persianblue}{rgb}{0.11, 0.22, 0.73}
\definecolor{rufous}{rgb}{0.66, 0.11, 0.03}
\definecolor{uablue}{rgb}{0.0, 0.2, 0.67}
\definecolor{zaffre}{rgb}{0.0, 0.08, 0.66}
\definecolor{carmine}{rgb}{0.59, 0.0, 0.09}

\usepackage[colorlinks, linkcolor=carminered, citecolor=electricultramarine, urlcolor=mahogany, hypertexnames=false]{hyperref}

\usepackage{fancyhdr}
\usepackage{lipsum}

\fancypagestyle{plain}{%
\fancyhf{} 
\fancyhead[C]{\thepage} 
}

\usepackage{esint} 
\usepackage{amssymb} 
\usepackage{stmaryrd}
\usepackage{cite} 

\usepackage{tikz}

\usepackage{mathtools}
\usepackage{soul}
\usepackage{color}

\usepackage{slashed}

\newtheorem{thm}{Theorem}[section]
\newtheorem{lem}{Lemma}[section]
\newtheorem{prop}{Proposition}[section]

\newtheorem{defi}{Definition}[section]
\newtheorem{rem}{Remark}[section]

\numberwithin{equation}{section}

\newcommand{\R}{\mathbb{R}}

\newcommand{\sQ}{\mathcal{Q}}

\newcommand{\ba}{\mathbf a}
\newcommand{\baa}{\mathbf A}
\newcommand{\ck}{\mathcal{K}_{\lambda}}

\usepackage{lipsum} 
\usepackage{tikz,xcolor}

\definecolor{lime}{HTML}{A6CE39}
\DeclareRobustCommand{\orcidicon}{
	\begin{tikzpicture}
	\draw[lime, fill=lime] (0,0) 
	circle [radius=0.16] 
	node[white] {{\fontfamily{qag}\selectfont \tiny ID}};
	\draw[white, fill=white] (-0.0625,0.095) 
	circle [radius=0.007];
	\end{tikzpicture}
	\hspace{-2mm}
}
\foreach \x in {A, B, C}{\expandafter\xdef\csname orcid\x\endcsname{\noexpand\href{https://orcid.org/\csname orcidauthor\x\endcsname}
			{\noexpand\orcidicon}}
}

\usepackage{bm}

\author[M. Hamouda, M. Majdoub \& T. Saanouni]{Makram Hamouda\orcidA{}, Mohamed Majdoub\orcidB{} \& Tarek Saanouni\orcidC{}}

\address[M. Hamouda]{Department of Basic Sciences, Deanship of Preparatory Year and Supporting Studies
Imam Abdulrahman Bin Faisal University, P. O. Box 1982, Dammam, Saudi Arabia}
\email{\sl \color{blue}{mmhamouda@iau.edu.sa}}
\email{\sl \color{blue}{mahamoud@iu.edu}}
\address[M. Majdoub]{Department of Mathematics, College of Science, Imam Abdulrahman Bin Faisal University, P. O. Box 1982, Dammam, Saudi Arabia.}
\address[M. Majdoub]{Basic and Applied Scientific Research Center, Imam Abdulrahman Bin Faisal University, P.O. Box 1982, 31441, Dammam, Saudi Arabia.}
\email{\sl \color{blue}{mmajdoub@iau.edu.sa}}
\email{\sl \color{blue}{mohamed.majdoub@fst.rnu.tn}}
\email{\sl \color{blue}{med.majdoub@gmail.com}}
\address[T. Saanouni]{Department of Mathematics, College of Science and Arts in Uglat Asugour, Qassim University, Buraydah, Kingdom of Saudi Arabia.}
\email{\sl \color{blue}{t.saanouni@qu.edu.sa}}

\title[Global existence  $\&$ Scattering]{Damping Effects on Global Existence and Scattering for an Inhomogeneous NLS Equation with Inverse-Square Potential}

\begin{document}


\begin{abstract}
This work explores the global existence and scattering behavior of solutions to a damped, inhomogeneous nonlinear Schr\"odinger equation featuring a time-dependent damping term, an inverse-square potential, and an inhomogeneous nonlinearity. We establish global well-posedness in the energy space for subcritical, mass-critical, and energy-critical regimes, using  Strichartz estimates, Hardy inequalities, and Gagliardo–Nirenberg-type estimates. For sufficiently large damping, we highlight how the interplay between damping, singular potentials, and inhomogeneity influences the dynamics. Our results extend existing studies and offer new insights into the long-time behavior of solutions in this more general setting. To the best of our knowledge, this is the first study to address the combined effects of inverse-square potential, inhomogeneous (or homogeneous) nonlinearity, and damping in the context of the NLS equation. 
\end{abstract}

\maketitle



\section{Introduction and main results}\label{intro}

We consider the Cauchy problem for the inhomogeneous Schr\"odinger (INLS) equation with inverse-square potential and an unsteady linear damping term, namely
\begin{equation}
\left\{
\begin{array}{ll}
{\rm i}\partial_t u-\ck\,u +{\rm i}\ba(t)u=\mu|x|^{-b}|u|^{p-1}u,\quad (t,x) \in \R\times\R^N,\vspace{.2cm}\\
u(0,x)=u_0(x),\quad x \in \R^N,
\label{DINLS}\tag{DINLS}
\end{array}
\right.
\end{equation}

where the space dimension is $N\geq3$, the inhomogeneous term exponent is $0<b<2$, the exponent of the source term is $p>1$, $\mu=\pm 1$ and the linear free operator is $\ck=-\Delta+\lambda|x|^{-2}$, where $\lambda>-\frac{(N-2)^2}{4}$.  The damping term is a time-dependent nonegative continuous function  $\ba:=\ba(t)\ge 0$. The focusing case corresponds to $\mu=-1$ and the defocusing one stands for $\mu=1$.

Based on the sharp Hardy inequality  \cite{abde}, 
\begin{equation}\label{Hardy}
\lambda_N\int_{\mathbb{R}^N}|x|^{-2} \, |f(x)|^2\,dx \leq \int_{\mathbb{R}^N}|\nabla f(x)|^2\,dx,
\end{equation}
where the optimal constant $\lambda_N$ is given by
\begin{equation}\label{lambda_N}
    \lambda_N = \frac{(N-2)^2}{4},
\end{equation}
it is known that the operator $-\Delta + \lambda|x|^{-2}$ admits a positive semi-definite symmetric extension $\ck$ for all $\lambda > -\lambda_N$. However, in the  range $-\lambda_N < \lambda < 1-\lambda_N$, this extension is not unique \cite{ksww,ect}. In such cases, one typically chooses the Friedrichs extension \cite{ksww,pst}, which is the canonical self-adjoint extension associated with the corresponding quadratic form:
\begin{equation}
    \label{quad-form}
    \left< \ck f, f\right>_{L^2}=\int_{\mathbb{R}^N}\bigg[|\nabla f(x)|^2+\lambda |x|^{-2} \, |f(x)|^2 \bigg]\,dx=\int_{\mathbb{R}^N}\Big|\nabla f(x)+\kappa |x|^{-2} \, xf(x) \Big|^2\,dx,
\end{equation}
for $\lambda > -\lambda_N$ and for any $f \in \mathcal{C}^{\infty}_c \left(\R^N \setminus \{0\}\right)$. In the above identity \eqref{quad-form}, we introduced the following parameter:
\begin{equation}
    \label{kappa}
    \kappa:=\sqrt{\lambda_N}-\sqrt{\lambda_N+\lambda}, \footnote{It is straightforward to see that 
$\kappa=\kappa(\lambda) : \left[-\lambda_N, \infty\right) \rightarrow \left(-\infty, \sqrt{\lambda_N}\,\right]$. }
\end{equation}
which is of essential use regarding the norms' equivalence. 

Note that  the operator $\mathcal K_{\lambda}$ is well-defined on the domain, 
\begin{equation*}
    H^1_{\lambda}(\R^N)=\{f \in H^1(\R^N): \|f\|_{H^1_{\lambda}} < \infty \},
\end{equation*}
where 
\begin{equation}
    \label{norm}
    \|f\|_{H^1_{\lambda}}^2=\|f\|_{L^2}^2+\|\nabla f\|_{L^2}^2+\int_{\mathbb{R}^N}\lambda |x|^{-2} \, |f(x)|^2\,dx = \|f\|_{L^2}^2+\|f\|_{\dot{H}^1_{\lambda}}^2.
\end{equation}
It is straightforward, by \eqref{Hardy}, to see that we have the equivalence of the norms as follows:
\begin{equation}
    \label{norm-equiv}
    \|f\|_{{H}^1_{\lambda}} \sim  \|f\|_{H^1}, \quad \text{for} \ \lambda > - \lambda_N.
\end{equation}

From the physical point of view, the equation \eqref{DINLS}  arises in diverse physical contexts ranging from quantum gases to nonlinear optics. Specifically, this model incorporates an inverse-square potential which arises in quantum three-body problems, a spatially modulated nonlinearity  that captures interaction inhomogeneities in trapped Bose-Einstein condensates, and a time-dependent damping term  modeling energy gain/loss dissipation. The interplay of these features yields a mathematically rich framework for analyzing global existence and scattering. For further details, see \cite{Belmonte, bpst,ksww,Kartashov} and the references therein.\\

 The mathematical analysis of \eqref{DINLS} presents a nontrivial synthesis of three fundamental phenomena: the critical scaling behavior induced by the inverse-square potential, the spatial modulation of nonlinear interactions, and the non-autonomous effects of damping. Each component individually raises delicate analytical questions concerning functional settings, dispersive properties, and long-time dynamics. The transition between these regimes reveals a subtle competition.\\

Equation \eqref{DINLS} with $\mathbf{a} = 0$ is invariant under the following scaling transformation:
\begin{equation*}
u_\delta(t, x) := \delta^{\frac{2 - b}{p - 1}} u(\delta^2 t, \delta x), \quad \delta > 0.
\end{equation*}

Associated with this scaling is the critical Sobolev exponent
\begin{equation}
\label{s-c}
s_c := \frac{N}{2} - \frac{2 - b}{p - 1},
\end{equation}
which leaves the homogeneous Sobolev norm invariant in the sense that
$$
\|u_\delta(t)\|_{\dot{H}^s} = \delta^{s - s_c} \|u(\delta^2 t)\|_{\dot{H}^s}.
$$

Two particular cases are of significant interest from a physical viewpoint. The first is the \textit{mass-critical} case, corresponding to $s_c = 0$, which occurs when $p =1 + \frac{4 - 2b}{N}$. This regime is associated with the conservation of mass. The second is the \textit{energy-critical} case, defined by $s_c = 1$, which arises when $p =1 + \frac{4 - 2b}{N - 2}$.\\

 In what follows, we provide an overview of the known results for both the undamped and damped cases.

In the undamped regime, where $\mathbf{a}(t) \equiv 0$ in equation~\eqref{DINLS}, extensive research has focused on several key special cases. The classical nonlinear Schrödinger (NLS) equation, arising when both $b = 0$ and $\lambda = 0$, has been studied in depth over the past three decades, with significant progress made on questions of well-posedness, soliton dynamics, and scattering theory~\cite{Caz-book,Tao-book,LP-book,Wang-book}. When $b = 0$ and $\lambda \neq 0$, the equation includes an inverse-square potential while retaining a homogeneous nonlinearity. This variant has attracted considerable attention in recent years~\cite{Dinh-18, HI-2025, Killip-17, K-Murphy-17, Lu-18, zz, Yang-21}. Another notable case arises when $\lambda = 0$ and $b \neq 0$, leading to an inhomogeneous NLS equation with spatially varying nonlinearity. This form has been widely investigated~\cite{An-21, Ardila-21, Campos-21, Dinh-18-NA, Guzman-17} due to its relevance in physical models such as nonlinear optics and Bose-Einstein condensates. The case $\lambda>-\lambda_N$  and $b>0$ has been examined in \cite{cg,Suzuki,sbts}.

More generally, the inhomogeneous nonlinear Schr\"{o}dinger equation with potential
\begin{equation*}\label{Gen-poten}
    i\partial_t u + \Delta u - Vu = \mu |x|^{-b}|u|^{p-1}u, 
\end{equation*}
has been the subject of active research. Dinh \cite{Dinh-21} established foundational results for this equation when $N=3$, $b>0$, and $\mu=\pm1$, considering potentials $V\in K_0\cap L^{3/2}$ that satisfy $\|V_-\|_{K} < 4\pi$, where the Kato norm $\|V\|_K := \sup_{x\in\mathbb{R}^3} \int_{\mathbb{R}^3} |V(y)||x-y|^{-1} dy$ characterizes the potential's regularity and $V_- := \min\{V,0\}$ represents its negative part. In the case of harmonic potentials, Luo \cite{Luo-19} investigated standing wave solutions for $V(x)=|x|^2$ with $\mu=-1$ and $b<0$, obtaining several stability and multiplicity results. \\

In the damped regime, we will now review the existing literature depending on the parameters $\lambda$ and $b$. 
The case $\lambda = b = 0$ has been extensively studied, with numerous works addressing local and global well-posedness, blow-up phenomena, and asymptotic behavior. The global versus blow-up has been widely studied in \cite{vdd, FZS14, HM-Damped NLS,ot3, Tsut1}. Regarding scattering, the constant damping case is investigated in \cite{vdd, inui}. For damping satisfying $t\ba(t)\sim \gamma>0$ as $t\to\infty$ and oscillating initial data, see \cite{Bamri,Tayachi-23}. Recently, energy scattering was obtained in \cite{MJM}.

However, to the best of our knowledge, the damped NLS with (non-trivial) potential is not studied in the literature.  In this paper, we are interested in the global existence and asymptotic behavior of \eqref{DINLS} in the general setting $\lambda \neq 0, b>0$ and $1<p \le 1+\frac{4-2b}{N-2}$.\\
 
Note that the local existence in $H^1$ to \eqref{DINLS}  is established in \cite{Suzuki} by energy methods under the assumptions that
\begin{equation}
    \label{assump-suz}
    N \ge 3, \quad 0<b<2, \quad 1<p<1+\frac{4-2b}{N-2} \quad \text{and} \quad \lambda > -\lambda_N.
    \end{equation}
However, it is not clear whether the local solution obtained by energy methods in \cite{Suzuki} belongs to the Strichartz spaces. The latter information is an essential tool in the long-time behavior analysis among other uses. In \cite{cg}, this additional property is proved assuming some conditions. In fact, the literature on the local well-posedness (LWP) theory in the energy space $H^1(\R^N)$ ($N\ge 3$) for the undamped NLS corresponding to \eqref{DINLS} ($\ba(t)=0$) can be summarized in Table \ref{tab_LWP}. 
\small{
\begin{table}[h]
    \centering
    \begin{tabular}{|c|c|c|c|}
        \hline
         $\lambda$ & $b$ & $p$& \bf{{Ref.}} \\
        \hline
       $\lambda \ge -\lambda_N$ & $0<b<2$ & $1<p<1+\frac{4-2b}{N-2}$ & \cite{Suzuki2012, Suzuki2012-bis, Suzuki} \\
        \hline
         $\lambda > -\lambda_N$ & $0\le b<1$ & $1+\frac{2-2b}{N}<p<1+\frac{2-2b}{N-2}$ & \cite{cg} \\
        \hline
        $\lambda > -\lambda_N +\left(\frac{(p-1)(N-2)-2+2b}{2p}\right)^2$& $0\le b<\min(N/2,2)$ & $1+\frac{2(1-b)_{+}}{N-2}<p<1+\frac{4-2b}{N-2}$ & \cite{cg} \\
        \hline
        $\lambda > -\frac{(N+2-2b)^2-4}{(N+2-2b)^2}\lambda_N$ & $0< b<4/N$ &$p=1+\frac{4-2b}{N-2}$ & \cite{Kim} \\
        \hline 
    \end{tabular}
    \vspace{.2cm}
    \caption{\small{LWP for \eqref{DINLS} with $\ba(t)=0$ in the energy space.}}
    \label{tab_LWP}
\end{table}}

It is shown in \cite{Suzuki2012, Suzuki2012-bis, Suzuki}, via the energy methods, the existence of a unique weak solution $u\in C([0,T); H^1)\cap C^1([0,T); H^{-1})$ to \eqref{DINLS} with $\ba(t)=0$. To the best of our knowledge, there is no available proof in the literature showing that the above-mentioned local solution belongs to any Strichartz space. Furthermore, the global existence is proved for $1<p<1+\frac{4-2b}{N-2}$ and $\mu=1$ or $1<p<1+\frac{4-2b}{N}$ and $\mu=-1$.

Using Strichartz estimates and a fixed-point argument, it is shown in \cite{cg,Kim} that \eqref{DINLS} with $\ba(t)=0$ has a unique maximal solution $u$ such that 
\begin{equation}
\label{u-space}
u\in C([0,T^*); H^1)\cap L^q_{loc}([0,T^*); W^{1,r}),
\end{equation}
for any admissible pair $(q,r)$ in the sense of \eqref{adm} below. Moreover, we have the following blow-up criterion:
\begin{equation}
\label{Blow-Crit}
\text{if} \quad T^*<\infty \quad \text{then} \quad\displaystyle\lim_{t\to T^*}\|\nabla u(t)\|_{L^2}=\infty. 
\end{equation}

Recent work in~\cite{cg} establishes sufficient conditions for global existence and blowup phenomena for $N \geq 3$, $\mu = -1$, and $1 + \frac{4 - 2b}{N} \leq p < 1 + \frac{4 - 2b}{N - 2}$, employing Gagliardo-Nirenberg-type estimates. This framework further yields small-data global well-posedness in the energy-subcritical case $p < 1 + \frac{4 - 2b}{N - 2}$ under some conditions on $b$ and $\lambda$, achieved through a synthesis of Strichartz estimates and a fixed-point argument. The analysis extends to scattering criteria and wave operator construction in $H^1$ for the intercritical regime. These developments complement prior investigations~\cite{cg, Suzuki, Suzuki2012, Suzuki2012-bis}, which systematically characterize local/global well-posedness, blowup dynamics, and scattering behavior in $H^1$ for $N \geq 3$, in the energy-subcritical case of \eqref{DINLS} with $\ba(t) = 0$.

In the energy-critical regime for \eqref{DINLS}, the local well-posedness as well as small data global well-posedness and scattering are proven in \cite{Kim}.

Using a change of unknowns $v(t,x)={\rm e}^{\baa(t)}u(t,x)$, where $\baa(t)$ is given by \eqref{A} below, and the fact that $t \mapsto {\rm e}^{(1-p)\baa(t)}$ is bounded, the local existence in the sense of \eqref{u-space}-\eqref{Blow-Crit} holds true for \eqref{DINLS-v} below, and consequently for the damped NLS \eqref{DINLS}.\\


\subsection{Main results} In this section, we present our main results on global existence and scattering for the equation \eqref{DINLS}. We consider only the focusing case, corresponding to $\mu=-1$. 

We begin with the following theorem which establishes global existence and scattering in both the mass-subcritical and the mass-critical settings.
\begin{thm}
\label{scat1}
Let $\lambda >-\lambda_N, \ 0 \le b <2$, $\displaystyle 1<p\leq 1+\frac{4-2b}{N}$, and $u_0 \in H^1(\R^N)$. Assume that $\underline{\ba}>0$.
\begin{itemize}
        \item[(a)] If $\displaystyle p< 1+\frac{4-2b}{N}$, then the solution $u$ to \eqref{DINLS}  exists globally and scatters in $H^1(\R^N)$.
        \item[(b)] If $\displaystyle p= 1+\frac{4-2b}{N}$ and \begin{equation}
            \label{Grd-ass}
            \|u_0\|<\left(\frac{N}{N+2-b}\right)^{\frac{N}{4-2b}}\|\sQ\|,
        \end{equation} then the solution $u$ to \eqref{DINLS}  exists globally and scatters in $H^1(\R^N)$. Here, $\sQ$ is the ground state given by \eqref{grd} below.
\end{itemize}
\end{thm}

\begin{rem}
\label{Rem2.3}
~\begin{itemize}{\rm
\item[$(i)$] Assumption \eqref{Grd-ass} imposes a stronger condition than the classical ground state criterion $\|u_0\| < \|\sQ\|$. 
\item[$(ii)$] For $\displaystyle 1<p<1+\frac{4-2b}{N}$, the global existence can be proven as in \cite[Proposition B.1]{HM-Damped NLS}.
\item[$(iii)$] In the mass-critical regime $\displaystyle p=1+\frac{4-2b}{N}$, the global existence is shown in \cite[Theorem 1.6]{HM-Damped NLS} for $\lambda=b=0$ and for $\underline{\ba}$ being sufficiently large. Under Assumption \eqref{Grd-ass}, global existence is obtained without any restriction on the damping term.
\item[$(iv)$] As mentioned before, for example, in \cite[Chapter 7, page 211]{Caz-book} and \cite[Section 3.6, page 162]{Tao-book}, scattering requires the nonlinearity to exhibit superlinear growth near zero. Specifically, the nonlinearity exponent must exceed a critical value, $p_*=p_*(N, \lambda, b)$, which marks the threshold between scattering and the nonexistence of asymptotically free solutions. 
\item[$(v)$] For the standard NLS ($\lambda=b=0$ in \eqref{DINLS}), the critical value is $p_*(N,0,0)=1+\frac{2}{N}$; see \cite{Barab,  Strauss2, Tao-2004, Tsut-Yaj}. Recently, in \cite{BS-2024} it was shown that $p_*(N,0,b)=1+\frac{2-2b}{N}$ for $0<b<1$. Alternatively, a partial result was obtained in \cite{Xia} for $N=3,\lambda>0, b=0$, that is, $p_*(3,\lambda,0)=1+\frac{2}{3}$. Based on the aforementioned recall, we conjecture that $p_*(N, \lambda, b)=1+\frac{2-2b}{N}$ for $\lambda>0$ and $0\le b<1$. 
\item[$(vi)$] We conjecture that the scattering result in Theorem \ref{scat1} remains valid even when the damping function is merely nonnegative for sufficiently large times.}
    \end{itemize}
\end{rem}

The following theorem deals with the global existence and scattering of \eqref{DINLS} in the inter-critical regime.
\begin{thm}\label{t1}
Let $N \ge 3$, $\lambda \ge 0$, $0<b<1$ and $1+\frac{4-2b}{N}<p<1+\frac{4-2b}{N-2}$.  For $u_0 \in H^1(\R^N)$ radial, there exists $\ba^*>0$ such that for all $\underline{\ba} \ge \ba^*$, where $\underline{\ba}$ is defined in \eqref{ais}, the solution to \eqref{DINLS}  exists globally and scatters in $H^1(\R^N)$.
\end{thm}

While the energy-critical case presents significant challenges for the undamped NLS (\eqref{DINLS} with $\ba(t)=0$), the inclusion of the damping term facilitates scattering, as established in the following theorem.
\begin{thm}
\label{LWP}
Let $N\ge3, \ 0<b<\frac{4}{N}, \ p=\frac{N+2-2b}{N-2}$, $r= \frac{2N(N+2-2b)}{N^2-2Nb+4}$ and $\lambda > -\frac{(N+2-2b)^2-4}{(N+2-2b)^2}\lambda_N$ with $\lambda_N$ given by \eqref{lambda_N}. For $u_0 \in H^1(\R^N)$, there exists $\ba^{**}>0$ such that for all $\underline{\ba} \ge \ba^{**}$, the  problem \eqref{DINLS} has a  unique global solution verifying
\begin{equation}\label{u-sol}
    u \in C(\R_+; H^1(\R^N)) \cap L^{\alpha(r)}(\R_+; W^{1,r}(\R^N)).
\end{equation}
Moreover, the solution $u$ satisfies
\begin{equation}
    \label{est-global}
    \|u\|_{L^{\alpha(r)}((t_0, +\infty);\ W^{1,r})}\leq C {\rm e}^{-\underline{\ba}t_0}\|u_0\|_{H^1},
\end{equation}
where $\left(\alpha(r),r\right)$ is an admissible pair in the sense of \eqref{adm} and $t_0>0$ is arbitrary chosen. Furthermore, the solution $u$ scatters.
\end{thm}
\begin{rem}
\label{Rem2.4}
~\begin{itemize}{\rm 
\item[$(i)$] As it will be shown later, the constant $\ba^{**}$ will be made explicit as follows:
    \begin{equation}
    \label{GWP-cond}
\ba^{**}:=\frac{1}{1-p} \left[\ln\left((2\mathbf{c}_0)^{-p} \|u_0\|_{H^1}^{1-p}\right)\right]_{-},
\end{equation}
where $\mathbf{c}_0$ is the best constant in the Strichartz estimates and $\xi_{-}=\min \{\xi,0\}$.
\item[$(ii)$] The above theorem highlights the influence of the damping term on the scattering behavior in the energy-critical regime.
\item[$(iii)$] Note that, among many other interesting results, the global existence and scattering for the critical value $\displaystyle p=1+\frac{4}{N-2}$ in the defocusing case is studied in the seminal work \cite{CKSTT-08} for $N= 3$, $\lambda=b=0$ and $\mu=1$.
\item[$(iv)$]  It is worth noting that the damping term forces the solution $u$ to exhibit an exponential scattering in all the aforementioned regimes, in the sense of \cite[Definition 1.1]{inui}.}
\end{itemize}
\end{rem}

The structure of the article is as follows. In Section \ref{aux}, we lay out the foundational framework for analyzing equation \eqref{DINLS}, along with several auxiliary results that will support our main arguments. Section \ref{mass} is devoted to the proof of Theorem \ref{scat1}, which addresses the sub-mass-critical case. In Section \ref{inter}, we turn to the proof of Theorem \ref{t1}, focusing on the inter-critical regime. The energy-critical case, as stated in Theorem \ref{LWP}, is established in Section \ref{crit}. Finally, we conclude the article in Section \ref{cr} with some final remarks and a discussion of open questions.
\section{Background and auxiliary results}\label{aux}
This section is dedicated to introducing the necessary notation, along with several auxiliary results and useful estimates.

For simplicity, we will employ the following notations: 
\begin{equation*}
\|\cdot\|_{L^r}:=\|\cdot\|_{L^r(\R^N)},\quad \|\cdot\|:=\|\cdot\|_{L^2},\quad\|\cdot\|_{H^1}:=\|\cdot\|_{H^1(\R^N)}.
\end{equation*}
Then, we define some energy-related quantities associated with equation \eqref{DINLS}, namely
\begin{eqnarray}
{\mathbf M}[u(t)]&:=& \int_{\R^N}|u(t,x)|^2\,dx, \label{M-u}\\
{\mathbf E}[{u(t)}]&:=& \int_{\R^N} |\nabla u(t,x)|^2 \,dx + \lambda \int_{\R^N} \frac{|u(t,x)|^2}{|x|^2} \,dx - \frac{2}{p+1} \int_{\R^N} |x|^{-b}|u(t,x)|^{p+1}\,dx, \label{E-u}\\
 {\mathbf I}[u(t)]&:=& \int_{\R^N} |\nabla u(t,x)|^2 \,dx + \lambda \int_{\R^N} \frac{|u(t,x)|^2}{|x|^2} \,dx  -\int_{\R^N} |x|^{-b}|u(t,x)|^{p+1}\,dx.\label{I-u}
\end{eqnarray}
Observe that the energy can be expressed as
\begin{equation}
    \label{E-uu}
    {\mathbf E}[{u(t)}]= \|\sqrt{\ck} u(t)\|^2 - \frac{2}{p+1} \int_{\R^N} |x|^{-b}|u(t,x)|^{p+1}\,dx.
\end{equation}

Next, we introduce the following notation:
\begin{gather}
\label{A}
\baa(t)=\int_0^{|t|}\,\ba(s)ds, \quad t \in \R,\\
\label{ais}
\underline{\ba}=\inf_{t>0}\,\bigg(\frac{\baa(t)}{t}\bigg), \quad \overline{\ba}=\sup_{t>0}\,\bigg(\frac{\baa(t)}{t}\bigg).
\end{gather}
To the best of our knowledge, the quantities defined in \eqref{A}–\eqref{ais} were first introduced in \cite{HM-Damped NLS}, within the context of the damped nonlinear Schr\"odinger equation.

Now, denoting by $U_{\ba, \lambda}(t)$ the free propagator associated with \eqref{DINLS}, one can easily verify that
\begin{equation}
\label{U-a}
U_{\ba,\lambda}(t)={\rm e}^{-\baa(t)}\,U_{\lambda}(t),
\end{equation}
where $\baa(t)$ is given by \eqref{A} and 
\begin{equation}
\label{U-lambda}
    U_{\lambda}(t):=e^{-it\mathcal K_\lambda}, \quad (\ck=-\Delta+\lambda|x|^{-2}).
\end{equation}
It is then quite classical that the Cauchy problem for \eqref{DINLS} can be written in an integral form (see \cite{Caz-book}):
\begin{equation}
\label{Integ-Eq}
u(t)=U_{\ba,\lambda}(t)u_0 + i \Phi_{\ba, \lambda}\Big[|\cdot|^{-b}|u(\tau)|^{p-1}u(\tau)\Big](t),
\end{equation}
where 
\begin{equation}
\label{Phi}
\Phi_{\ba, \lambda}[f](t):= \int_0^t\,U_{\ba,\lambda}(t-\tau) f(\tau)\,d\tau.
\end{equation}

The following proposition presents two well-known identities related to the quantities defined in \eqref{M-u}–\eqref{I-u}. 
\begin{prop}
\label{Rela}
Let $u$ be a sufficiently smooth solution of \eqref{DINLS}  on $0\leq t\leq T$. Then, we have
\begin{align}
\label{M-Id}
&{\mathbf M}(u(t))={\rm e}^{-2\baa(t)}{\mathbf M}(u_0),\vspace{.3cm}\\
\label{E-Id}
&\frac{d}{dt}{\mathbf E}(u(t))=-2\ba(t) {\mathbf I}(u(t)).
\end{align}
\end{prop}
The proof of Proposition \ref{Rela} can be derived by suitably adapting the argument from \cite[Lemma 1]{Tsut1}.\\

Taking the transformation $v(t,x):=e^{\baa(t)}u(t,x)$, we are reduced to study the equivalent problem

\begin{equation}
\left\{
\begin{array}{ll}
i\partial_t v-\mathcal{K}_{\lambda}v=- e^{-(p-1)\baa(t)}|x|^{-b}|v|^{p-1}v,\\
v(0,x)=u_0(x),
\label{DINLS-v}
\end{array}
\right.
\end{equation}
where $\baa(t)$ is given by \eqref{A}.\\
Note that the global existence for $v$
automatically guarantees the same for $u$, and conversely.

In the following lemma we give the property of the equivalence of norms for the operator $\mathcal{K}_{\lambda}$ in Lebesgue spaces. For further readings, see \cite[Theorem 1.2]{kmvzz}.
\begin{lem}
    \label{equiv-norm-leb}
    Let $N \ge 3$ and $\lambda \ge -\lambda_N$. Assume either $``\lambda >0 \ \text{and} \ 1<r<N"$ or $``\lambda <0 \ \text{and} \ \frac{N}{N-\kappa}<r<\frac{N}{1+\kappa}"$, where $\kappa$ is defined by \eqref{kappa}.  Then, we have
    \begin{equation}
        \label{equiv-norm-id}
        \|\sqrt{\mathcal{K}_{\lambda}} f \|_{L^r} \lesssim_{N,r} \|\nabla f \|_{L^r}\lesssim_{N,r} \|\sqrt{\mathcal{K}_{\lambda}} f \|_{L^r}, \quad \text{for all} \ f \in \mathcal{C}^{\infty}_c (\R^N).
    \end{equation}
\end{lem}

Before stating the Strichartz estimates, the next definition \eqref{adm} and observation \eqref{adm-s} are required.
\begin{defi}\label{df}
Let $N\ge 3$. A couple of real numbers $(q,r)$ is said to be  admissible if 
\begin{equation}\label{adm}
    \frac{2}{q}=N\left(\frac12-\frac1r \right), \quad 2\le r \le \frac{2N}{N-2}.
\end{equation}
    \end{defi}
\begin{rem}
{\rm Note that the above definition can be extended to the so-called 
     $s-$admissible statement, namely a couple of real numbers $(q,r)$ is said to be  $s-$admissible if 
\begin{equation}\label{adm-s}
    \frac{2}{q}=N\left(\frac12-\frac1r \right)-s, \quad \frac{2N}{N-2s}<r<\frac{2N}{N-2}, \quad s>0.
    \end{equation}}
\end{rem}

Now, in the next lemmas, we will state the Strichartz estimates \cite{bpst,cg,df, msz}.
\begin{lem}\label{str}
Let $N\geq3$, $\lambda>-\lambda_N$ and $0\in I$ be a real interval. Assume further that $(q,r)$ and $(\tilde{q}, \tilde{r})$ are two admissible pairs in the sense of \eqref{adm}. Then, there exists $C>0$ such that 
\begin{equation}
    \label{str1}
\|e^{-i t \mathcal K_\lambda}f\|_{L^q(I; L^r)}\leq C\|f\|_{L^2},
\end{equation}
and
\begin{equation}\label{str2}
\left\|\int_0^{t}e^{-i(t-\tau)\mathcal K_\lambda}g(\tau)\,d\tau\right\|_{L^q(I; L^r)}\leq C\|g\|_{L^{\tilde{q}'}(I; L^{\tilde{r}'})}.
\end{equation}
\end{lem}

\begin{lem}\label{str-s}
Let $N\geq3, s>0$, $\lambda>-\lambda_N$ and $0\in I$ be a real interval. Assume that $(q,r)$ be an $s-$admissible pair in the sense of \eqref{adm-s}. Then, there exists $C>0$ such that 
\begin{equation}
    \label{str3}
\|e^{-i t \mathcal K_\lambda}f\|_{L^q(I; L^r)}\leq C\|f\|_{\dot{H}^s_{\lambda}}.
\end{equation}
\end{lem}

\begin{lem}
    \label{stri-radial}
Let $N\geq3, s>0$, $\lambda\geq0$, $g$ being spherically symmetric, and $0\in I$ be a real interval. Assume that $(q,r)$ and $(\tilde{q}, \tilde{r})$ are two $s-$admissible pairs in the sense of \eqref{adm-s}. Then, we have
\begin{equation}
    \label{stri-radial-est}
    \left\|\int_0^{t}e^{-i(t-\tau)\mathcal K_\lambda}g(\tau)\,d\tau\right\|_{L^q(I; L^r)}\leq C\|g\|_{L^{\tilde{q}'}(I; L^{\tilde{r}'})}.
\end{equation}
\end{lem}
We will denote by $\mathbf{c}_0$ the sharp constant for the Strichartz estimate \cite{Alless-EJDE}.

One of the key ingredients in establishing the above Strichartz estimates is the dispersive inequality, stated below, as presented in \cite[Corollary 1.7]{fffp}.
\begin{prop}\label{dsp}
Let $N\ge 2$, $\lambda \ge -\lambda_N$ and $\kappa$ as defined in \eqref{kappa}. Assume that 
\begin{equation}
    \label{cond-r}
     2 \le r < \frac{N}{\kappa_+}\footnote{Hereafter we use the notation $\varsigma_+ :=\max(\varsigma,0)$ with the convention $0^{-1}=\infty$.} \quad \text{or} \quad 2 \le r \le \infty \quad \text{if} \quad \lambda =0.
\end{equation}
Then, we have
\begin{equation}
    \label{dis-est}
    \|U_{\lambda}(t)\varphi\|_{L^{r}}\lesssim |t|^{-N(\frac12-\frac1{r})} \|\varphi\|_{L^{r'}},
\end{equation}
where $U_{\lambda}(t)$ is given by \eqref{U-lambda}.
\end{prop}

\begin{rem}
\rm
~
\begin{itemize}
    \item[($i$)] For the case $\lambda =0$, the  dispersive estimate \eqref{dis-est} is known for all $2 \le r \le \infty$ since it concerns the free Schr\"odinger propagator.
    \item[(ii)] In the case $\lambda \neq 0$, the $L^{\infty}-L^1$ dispersive estimate is known to fail at least for $\lambda < 0$; see \cite[Remark 1, p. 530]{bpst}. However, in the radial setting, the $L^{\infty}-L^1$ dispersive estimate holds true for all $\lambda \ge 0$; see \cite[Theorem 2.8]{zheng-18}.
    \item[(iii)] For $-\lambda_N < \lambda < 0$ and $\kappa$ as given by \eqref{kappa},   a weighted $L^{\infty}-L^1$ dispersive estimate is obtained as follows \cite[Theorem 2.8]{zheng-18}:
\begin{equation}
    \label{dis-est-3d}
    \left\|(1+|x|^{-\kappa})^{-1}U_{\lambda}(t)\varphi \right\|_{L^{\infty}(\R^N)}\lesssim  t^{-\frac{N}2} (1+|t|^\kappa)\left\|(1+|x|^{-\kappa}) \varphi\right\|_{L^{1}(\R^N)},
\end{equation}
provided that $\varphi$ is spherically symmetric. 
 \item[(iv)] In all cases we have $\frac{2N}{N-2}<\frac{N}{\kappa_+}$.
  \item[(v)]  Note that for $\alpha, \beta>0$,
\begin{equation}
\label{Oper-power}
{\mathcal K_{\lambda}}^\alpha\,{\mathcal K_{\lambda}}^\beta={\mathcal K_{\lambda}}^{\alpha+\beta};
\end{equation}
see e.g. \cite{Martinez} for the precise conditions about the validity of \eqref{Oper-power}.
\end{itemize}
\end{rem}

Now, we state a refined Strichartz estimate which is useful in our analysis.
\begin{lem}
\label{strichartz}
Let $0<T\le \infty, \ 2 <r<\frac{2N}{N-2}$ and $(\theta, \tilde{\theta}) \in (1, \infty)$ such that
\begin{equation}
\label{theta-thetatild}
\frac{1}{\theta} + \frac{1}{\tilde{\theta}} = N \left(\frac12-\frac1r\right).
\end{equation} 
Then, one has
\begin{equation}
\label{lorentz-est}
\|\Phi_{\ba, \lambda}[f](t)\|_{L^{r}} \ \le \ C \int_0^T |t-\tau|^{-N \left(\frac12-\frac1r\right)}\|f(\tau)\|_{L^{r'}} \, d\tau,
\end{equation}
where $C=C(r,\theta,N)$ is a positive constant.
\end{lem}

\begin{rem}
    \rm
~
Let $N\ge 3,\, 0<b <2$ and $ 1+\frac{4-2b}{N}<p<1+\frac{4-2b}{N-2}$. Define 
\begin{equation}
\label{theta-r-q}
\left\{
\begin{aligned}
    \theta &= \frac{2(p+1)(p-1)}{4 - 2 b - (N-2)(p-1)},\\
    r &= \frac{p+1}{1-\frac{b}{N}}, \qquad q =\frac{4(1+p)}{2 b+N(p-1)}.
    \end{aligned}
    \right.
\end{equation}
One can easily see that $\theta>\max(1, p-1)$, $2<r<\frac{2N}{N-2}$, $(q,r)$ is admissible in the sense of \eqref{adm}. Furthermore, the pair $(\theta,r)$ is $s_c-$admissible following \eqref{adm-s}, where $s_c$ is given by \eqref{s-c}. 
Moreover, by \eqref{theta-thetatild}, we have
\begin{equation}
    \label{p-theta}
    \tilde{\theta}' (p+1)=\theta,
\end{equation}
and
\begin{equation}
    \label{q-theta}
    \frac{1}{q'}=\frac{1}{q}+\frac{p-1}{\theta},
\end{equation}
where $\tilde{\theta}$ is given by \eqref{theta-thetatild}.
\end{rem}

As in \cite{Guzman-17}, we introduce the following numbers that will be useful in the proof of global existence for small data, namely
\begin{equation}
    \label{q-r-star}
    \begin{split}
    q_* &= \frac{4(p-1)(p+1-\eta)}{(p-1)(N(p+1)+2b-2N)-\eta(N(p+1)+2b-2N-4)}, \\
    r_* &= \frac{N(p-1)(p+1-\eta)}{(p-1)(N-b)-\eta(2-b)},
    \end{split}
\end{equation}
where $0<\eta\ll1$ is small enough to be chosen later on.\\
Note that the pair $(q_*, r_*)$ is admissible in the sense of Definition \ref{df}. We borrow without proof from \cite{cg} (using our notations) the following lemma.
\begin{lem} (\cite[Lemma 4.4]{cg})
    \label{lem-4.4}
    Let $N\ge 3$, $0<b<\min(N/2,2)$, $\lambda >0$, $1+\frac{4-2b}{N}<p<1+\frac{4-2b}{N-2}$ and $(q,r, \theta)$ as defined in \eqref{theta-r-q}. Then, there exists $\eta \in (0,p-1)$ such that
    \begin{itemize}
        \item[$(i)$]  $\left\|\sqrt{\mathcal{K}_\lambda} \left(|x|^{-b} |u|^{p-1}u\right) \right\|_{L^{q_*'}(L^{r_*'})} \lesssim  \|u\|^\eta_{L^\infty (H^1_\lambda)} \|\sqrt{\mathcal{K}_\lambda} u\|_{L^q (L^r)} \|u\|^{p-1-\eta}_{L^\theta (L^r)}$ for $N\ge 4$.$\vspace{.2cm}$
        \item[$(ii)$]  $\left\|\sqrt{\mathcal{K}_\lambda} \left(|x|^{-b} |u|^{p-1}u\right) \right\|_{L^{2}(L^{6/5})} \lesssim  \|u\|^\eta_{L^\infty (H^1_\lambda)} \|\sqrt{\mathcal{K}_\lambda} u\|_{L^q (L^r)} \|u\|^{p-1-\eta}_{L^\theta (L^r)}$ for $N=3$ and $p<4-2b$.
    \end{itemize}
\end{lem}

The following Sobolev embedding will be useful to our purpose:
\begin{equation}
    \label{Sob-Lore-Emb}
    H^s(\R^N) \hookrightarrow L^{r}(\R^N), \,\, s>0,\,\, \frac{1}{2}-\frac{s}{N}\leq \frac{1}{r}\leq \frac{1}{2},\,\, r<\infty.
\end{equation}
In what follows, we present the continuity (or bootstrap) argument, which will be useful in subsequent analysis.
\begin{lem}\cite[Lemma 3.7, p. 437]{Strauss}
\label{boots}
Let $I\subset\R$ be a time interval, and $\mathbf{X} : I\to [0,\infty)$ be a continuous function satisfying, for every $t\in I$,
\begin{equation*}
		\label{boots1}
		    	\mathbf{X}(t) \leq a + b [\mathbf{X}(t)]^\theta,
		\end{equation*}
	where $a,b>0$ and $\theta>1$ are constants. Assume that, for some $t_0\in I$,
		\begin{equation*}
		\label{boots2}
	\mathbf{X}(t_0)\leq a, \quad a\,b^{\frac{1}{\theta-1}} <(\theta-1)\,\theta^{\frac{\theta}{1-\theta}}.
				\end{equation*}
		Then, for every $ t\in I$, we have
		\begin{equation*}
		\label{boots3}
		    	\mathbf{X}(t)\leq  \frac{\theta\,a}{\theta-1}.
		\end{equation*}
\end{lem}

We will also make essential use of the following Gagliardo-Nirenberg inequality \cite[Lemma 3.2]{cg}. The similar version for $\lambda=0$  has been established in various forms in the literature \cite{Far, Genoud2012, pz}.
\begin{lem}\label{gag} \cite[Lemma 3.2]{cg}) 
Let $N\geq3$, $0<b<2$, $\lambda>-\lambda_N$ and $1<p<1+\frac{2(2-b)}{N-2}$. Then, it holds that
\begin{align}\label{ineq}
\int_{\R^N}|u(x)|^{p+1}|x|^{-b}\,dx\leq \mathtt{K}_{opt} \|u\|^{p+1-\nu}\|\sqrt{\mathcal K_\lambda} u\|^{\nu},\quad \text{for all} \ u\in H^1,
\end{align}
where  $\nu$ is given by 
\begin{equation}
\label{AB}
\nu:=\frac{N}{2}(p-1)+b.
\end{equation}
Moreover, the sharp constant $\mathtt{K}_{opt}$ is given by
\begin{equation}\label{part3}
\mathtt{K}_{opt} =\frac{p+1}{p+1-\nu}\Big(\frac{p+1-\nu}{\nu}\Big)^\frac{\nu}2\|\sQ\|^{-(p-1)},
\end{equation}
where $\sQ$ is a ground state solution to 
\begin{align}\label{grd}
\mathcal{K}_\lambda \sQ+|x|^{-b}|\sQ|^{p-1}\sQ=\sQ,\quad 0\neq\sQ\in H^1_\lambda.
\end{align}
Furthermore, the Pohozaev type identities hold
\begin{align}
    \|\sqrt{\mathcal K_\lambda}\sQ\|^2=  \frac{\nu}{p+1-\nu}\|\sQ\|^2=\frac{\nu}{p+1}\int_{\R^N} |\sQ(x)|^{p+1} |x|^{-b}dx.\label{poh}
\end{align}
\end{lem}


We recall the following fractional Hardy inequality which will be used, among others, in proving the local existence solution of \eqref{DINLS}.
\begin{lem}\label{FHI}(\cite[Theorem 3.1]{Haj})
    Let $1<q<\infty$ and $0<s<\frac{N}{q}$. Then we have
\begin{equation}
    \left\|\frac{u}{|x|^{s}} \right\|_{L^q} \lesssim \|u\|_{\dot{W}^{s,q}}.
\end{equation}
In particular, we have
\begin{equation}
    \label{Hardy-Lorentz}
    \left\| \frac{u}{|x|}\right\|_{L^{q}} \lesssim \left\| \nabla u\right\|_{L^{q}},\quad\mbox{whenever}\quad 1<q<N.
\end{equation}
\end{lem}

As a direct application of Lemma \ref{FHI}, the following lemma provides useful estimates for the nonlinear term in \eqref{DINLS}.
\begin{lem}\label{NLE}(\cite[Lemma 3.2]{Kim})
    Let $N \ge 3, \ 0< b < \frac{4}{N}, \ p=\frac{N+2-2b}{N-2}, \ r=\frac{2N(N+2-2b)}{N^2-2b N +4}$ and $\sigma=\sigma(N):=\frac{2N}{N-2}$. Then the following inequalities hold true
    \begin{equation}\label{NLE1}
        \||x|^{-b} |u|^{p-1}u\|_{L^2(I;\dot{W}^{1,\sigma'}(\R^N))} \lesssim \|u\|^{p}_{L^{\alpha(r)}(I; \dot{W}^{1,r}(\R^N))},
    \end{equation}
    \begin{equation}\label{NLE2}
        \||x|^{-b} |u|^{p-1}v\|_{L^2(I; \dot{W}^{1,\sigma'}(\R^N))} \lesssim \|u\|^{p-1}_{L^{\alpha(r)}(I;\dot{W}^{1,r}(\R^N))} \|v\|_{L^{\alpha(r)}(I; W^{1,r}(\R^N))},
    \end{equation}
    where $I\subset \R$ is an interval and the pair $(\alpha(r),r)$ is admissible in the sense of \eqref{adm}.
\end{lem}
\begin{rem}
{\rm The proof of Lemma \ref{NLE} is based on the combination of \cite[Lemma 3.2]{Kim}, H\"older's inequality and the fact that $\frac{1}{\alpha(\sigma)'}=\frac{1}{2}=\frac{p}{\alpha(r)}$.}
\end{rem}

The following lemma will be essential to establish the completeness of a metric space that will be subsequently employed in our fixed point argument.
\begin{lem}
    \label{Complete}
    Let $1 < q, r < \infty$, let $I \subset \mathbb{R}$ be an interval, and let $M > 0$. Define the space
$$
\mathbf{Y} := \big\{ u \in L^q(I; W^{1,r}(\mathbb{R}^N)); \;\; \|u\|_{L^q(I; W^{1,r})} \leq M \big\},
$$
equipped with the metric
$$
d(u,v) := \|u - v\|_{L^q(I; L^r)}.
$$
Then $(\mathbf{Y}, d)$ forms a complete metric space.
\end{lem}
\begin{proof}[{Proof of Lemma \ref{Complete}}]
    Let $(u_n)_{n\in\mathbb{N}}$ be a Cauchy sequence in $(\mathbf{Y}, d)$. Since $L^q(I; L^r(\mathbb{R}^N))$ is a complete metric space with respect to the standard Lebesgue norm, there exists $u \in L^q(I; L^r(\mathbb{R}^N))$ such that 
\begin{equation}
    \|u_n - u\|_{L^q(I; L^r)} \to 0 \quad \text{as } n\to\infty.
\end{equation}
To establish the completeness of $(\mathbf{Y}, d)$, it suffices to show that $u\in \mathbf{Y}$.\\
First, observe that the strong convergence in Lebesgue spaces implies the convergence in the sense of distributions:
\begin{equation}
    \label{eq:conv_dist_u}
    u_n\to u \quad \text{in } \mathcal{D}'(I\times\mathbb{R}^N).
\end{equation}
Since $L^q(I; W^{1,r})$ is a reflexive Banach space (due to the fact $1<q,r<\infty$) and $(u_n)$ is bounded in this space by definition of $\mathbf{Y}$, there exists (by weak sequential compactness) a subsequence $(u_{n_k})_{k\in\mathbb{N}}$ that converges weakly to some $v\in L^q(I; W^{1,r})$. Moreover, by the lower semicontinuity of the norm under weak convergence \cite[Proposition 3.5 (iii), p. 58]{Brezis}, we have:
\begin{equation}
    \label{eq:bound_v}
    \|v\|_{L^q(I; W^{1,r})} \leq \liminf_{k\to\infty} \|u_{n_k}\|_{L^q(I; W^{1,r})} \leq M.
\end{equation}
The weak convergence in $L^q(I; W^{1,r}(\mathbb{R}^N))$ implies convergence in $\mathcal{D}'(I\times\mathbb{R}^N)$:
\begin{equation}
    \label{eq:conv_dist_v}
    u_{n_k} \to v \quad \text{in } \mathcal{D}'(I\times\mathbb{R}^N).
\end{equation}
From \eqref{eq:conv_dist_u} and \eqref{eq:conv_dist_v}, we deduce by uniqueness of distributional limits that $u = v$. Therefore, we conclude that $u \in L^q(I; W^{1,r})$ (inherited from $v$) with 
$\|u\|_{L^q(I; W^{1,r})} \leq M$ by \eqref{eq:bound_v}, 
which establishes that $u \in \mathbf{Y}$. This completes the proof of Lemma~\ref{Complete}.
\end{proof}
\section{Mass-sub-critical regime}\label{mass}
In this section, we prove Theorem \ref{scat1}.

\begin{proof}[Proof of Theorem \ref{scat1}]
We begin by defining the Hamiltonian for $v$:
\begin{equation}\label{Hamil-v}
    \mathbf{H}(v(t)) := \|\nabla v(t)\|^2  +\lambda \int_{\R^N} \frac{|v(t,x)|^2}{|x|^2} \,dx - \frac{2}{p+1}e^{-(p-1)\baa(t)}\int_{\R^N} |x|^{-b}|v(t,x)|^{p+1}\,dx.
\end{equation}
Through direct computation, we obtain its time derivative:
\begin{equation}\label{Hamil-v-dt}
    \frac{d}{dt}\mathbf{H}(v(t)) = \frac{2(p-1)}{p+1}\ba(t)e^{-(p-1)\baa(t)}\int_{\R^N}|x|^{-b} |v(t,x)|^{p+1}\,dx.
\end{equation}
Integration over time yields the identity:
\begin{equation}\label{Hamil-v-id}
    \mathbf{H}(v(t)) = \mathbf{E}(u_0) + \frac{2(p-1)}{p+1}\int_0^t \ba(s)e^{-(p-1)\baa(s)}\left(\int_{\R^N} |x|^{-b}|v(s,x)|^{p+1}\,dx\right)ds.
\end{equation}
Similarly, we define the Hamiltonian for $u$ as:
\begin{equation}\label{tildeH}
    \mathcal{H}(u(t)) := e^{2\baa(t)}\mathbf{E}(u(t)) - \frac{2(p-1)}{p+1}\int_0^t \ba(s)e^{2\baa(s)}\left(\int_{\R^N} |x|^{-b}|u(s,x)|^{p+1}\,dx\right)\,ds.
\end{equation}
From \eqref{E-Id}, we immediately deduce the conservation property:
\begin{equation}\label{Ham-u}
    \frac{d}{dt}\mathcal{H}(u(t)) = 0.
\end{equation}
Combining \eqref{tildeH} and \eqref{Ham-u}, we obtain the energy identity:
\begin{equation}\label{H-Id}
    \mathbf{E}(u_0) = e^{2\baa(t)}\mathbf{E}(u(t)) - \frac{2(p-1)}{p+1}\int_0^t \ba(s) e^{2\baa(s)}\left(\int_{\R^N} |x|^{-b}|u(s,x)|^{p+1}\,dx\right)\,ds,
\end{equation}
where $\baa(t)$ and $\mathbf{E}(u(t))$ are defined in \eqref{A} and \eqref{E-u}, respectively. 

First, consider the mass-subcritical regime, that is $p<1+\frac{4-2b}{N}$.
Using the energy functional $\mathbf{E}(u(t))$ from \eqref{E-uu} together with \eqref{M-Id}, Hardy's inequality \eqref{Hardy} and the Gagliardo-Nirenberg inequality \eqref{ineq}, we obtain the lower bound:
\begin{equation}\label{LE-Energy'}
    \mathbf{E}(u(t)) \geq \|\sqrt{\mathcal{K}_\lambda} u(t)\|^2- \frac{2\mathtt{K}_{opt}}{1+p}\|u_0\|^{1+p-\nu}\|\sqrt{\mathcal{K}_\lambda} u(t)\|^{\nu},
\end{equation}
where the exponent $\nu$ is given by
\begin{equation}\label{nuu'}
    \nu = b + \frac{N(p-1)}{2}.
\end{equation}
Combining \eqref{H-Id} and \eqref{LE-Energy'}, we derive the following estimate:
    \begin{equation}
    \label{Pre-Gron'}
    \begin{split}
    \mathbf{E}(u_0) \geq &\ e^{2\baa(t)}\left(\|\sqrt{\mathcal{K}_\lambda} u(t)\|^2- \frac{2\mathtt{K}_{opt}}{1+p}\|u_0\|^{1+p-\nu}\|\sqrt{\mathcal{K}_\lambda} u(t)\|^{\nu}\right)\\
    &- \frac{2\mathtt{K}_{opt}(p-1)}{1+p}\|u_0\|^{1+p-\nu}\, \int_0^t \ba(s) e^{2\baa(s)} \|\sqrt{\mathcal{K}_\lambda} u(s)\|^\nu\,ds.
    \end{split}
    \end{equation}
We can rewrite \eqref{Pre-Gron'} in the form
\begin{equation}
\label{Gronw-new}
    f(t) \leq \mathbf{E}(u_0) + g(t)(f(t))^{\beta} + \int_0^t h(s)(f(s))^{\beta}\,ds, \quad \forall\,\,\; 0\leq t<T_{max},
\end{equation}
where we define:
\begin{align*}
    f(t) &= e^{2\baa(t)}\|\sqrt{\mathcal{K}_\lambda} u(t)\|^2, \\
    g(t) &= \frac{2\mathtt{K}_{opt}}{p+1}\|u_0\|^{p+1-\nu} e^{(2-\nu)\baa(t)}, \\
    h(t) &= \frac{2(p-1)\mathtt{K}_{opt}}{p+1}\|u_0\|^{p+1-\nu}\,\ba(t) e^{(2-\nu)\baa(t)}, \\
    \beta &= \frac{\nu}{2}=\frac{b}{2}+\frac{N(p-1)}{4}.
\end{align*}
For the range $p < 1 + \frac{4-2b}{N}$ (which implies that $\beta<1$), we verify that all conditions of \cite[Lemma 3.1]{MJM} are satisfied. Consequently, by norm equivalence \eqref{equiv-norm-id}, we obtain the uniform bound,
\begin{equation}
\label{GE-est}
    \sup_{0\leq t<T_{max}} e^{\baa(t)}\|\sqrt{\mathcal{K}_\lambda} u(t)\| < \infty.
\end{equation}
The above estimate together with the blow-up alternative yield the global existence, that is $T_{max}=\infty$. Furthermore, the estimate \eqref{GE-est} implies the global bound
\begin{equation}
\label{Scat-est}
   \|\sqrt{\mathcal{K}_\lambda}u(t)\|\lesssim e^{-\baa(t)}\lesssim e^{-\underline{\ba}t}, \quad \forall \ t\geq 0.
\end{equation}
Therefore, the scattering for $p < 1 + \frac{4-2b}{N}$ is obtained similarly as in \cite[Theorem 2.2]{MJM}.\\
Now, we consider the mass-critical case \( p = 1 + \frac{4-2b}{N} \). In this setting, we clearly have \( \nu = 2 \) and \( \mathtt{K}_{\text{opt}} = \frac{1+p}{2} \|\mathcal{Q}\|^{1-p} \). Thus, inequality \eqref{Pre-Gron'} becomes  
\begin{align}
    \mathbf{E}(u_0) 
    &\geq \left(1 - \left(\frac{\|u_0\|}{\|\sQ\|}\right)^{p-1}\right) \left(e^{\baa(t)}\|\sqrt{\mathcal{K}_\lambda} u(t)\|\right)^2 \nonumber \\
    &\quad - (p-1)\left(\frac{\|u_0\|}{\|\sQ\|}\right)^{p-1} \int_0^t \ba(s) \left(e^{\baa(s)}\|\sqrt{\mathcal{K}_\lambda} u(s)\|\right)^2 \, ds. \label{Pre-Gron''}
\end{align}  
Note that, by taking $t=0$ in the RHS of \eqref{Pre-Gron''} and using \eqref{Grd-ass}, one can easily see that $\mathbf{E}(u_0)\geq 0$. Now, recalling assumption \eqref{Grd-ass}, we obtain  
\begin{equation}
    \label{Pre-Gron3}
\left(e^{\baa(t)}\|\sqrt{\mathcal{K}_\lambda} u(t)\|\right)^2 
    \leq \frac{\mathbf{E}(u_0)}{1-\varrho} + \frac{(p-1)\varrho}{1-\varrho} \int_0^t \ba(s) \left(e^{\baa(s)}\|\sqrt{\mathcal{K}_\lambda} u(s)\|\right)^2 \, ds,
\end{equation}  
where \( \varrho := \left(\frac{\|u_0\|}{\|\mathcal{Q}\|}\right)^{p-1} \).  
Applying Gronwall's lemma to \eqref{Pre-Gron3} implies that  
\begin{equation}
    \label{Pre-Gron4}
    \|\sqrt{\mathcal{K}_\lambda} u(t)\|^2 
    \leq \frac{\mathbf{E}(u_0)}{1-\varrho} \, \exp\left(\left(\frac{(p-1)\varrho}{1-\varrho} - 2\right) \baa(t)\right).
\end{equation}  
Consequently, scattering holds provided  
\begin{equation}
    \label{cnd}
    \frac{(p-1)\varrho}{1-\varrho} < 2.
\end{equation}  
Indeed, \eqref{cnd} combined with the fact that $\underline{\ba}>0$ gives
\begin{equation}
    \label{Pre-Gron5}
    \|\sqrt{\mathcal{K}_\lambda} u(t)\|^2 
    \leq \frac{\mathbf{E}(u_0)}{1-\varrho} \, \exp\left(-\left(2-\frac{(p-1)\varrho}{1-\varrho}\right) \underline{\ba}t\right).
\end{equation}  
A direct computation reveals that \eqref{cnd} is equivalent to \eqref{Grd-ass}.
\end{proof}
\section{Inter-critical regime}\label{inter}
In this section, we establish the proof of Theorem \ref{t1} which is essentially based on the following proposition. 
\begin{prop}
    \label{prop3}
    Assume the hypotheses in Theorem \ref{t1} hold true. Let $\theta, r$ and $q$ as defined in \eqref{theta-r-q}. Suppose further that $\underline{\ba}>0$. Then, there exists $\varepsilon>0$ independent of $\ba(t)$ such that the solution of \eqref{DINLS} exists globally whenever 
    \begin{equation}
    \label{small-glob}
        \left\|U_{\ba,\lambda}(\cdot)u_0\right\|_{L^{\theta}(0,\infty; L^{r})} \le \varepsilon.
    \end{equation}
\end{prop}

\begin{proof}[{Proof of Proposition \ref{prop3}}]
Although we deal here with a time-dependent damping term, the proof of Proposition \ref{prop3} follows the same lines as in \cite[Theorem 1.8]{cg}. However, to facilitate the reading of the article, we will give the most important steps. First, recall the definition of $(q,r, \theta, \Tilde{\theta})$ as given by \eqref{theta-r-q} and \eqref{theta-thetatild}. Then, let us define
\begin{equation}
    \label{Phi-Duh}
    \Phi(u)(t)=U_{\ba,\lambda}(t)u_0 +\,i \int_0^t\,U_{\ba,\lambda}(t-\tau)\left[|\cdot|^{-b}|u(\tau)|^{p-1}u(\tau)\right]d\tau,
\end{equation}
\begin{equation}
    \label{X}
    \mathbf{X}=\left\{ u:\, \|u\|_{L^{\theta}(0,\infty; L^{r})} \le  2\varepsilon,\, \|u\|_{L^{\infty}(0,\infty; H^{1}_\lambda)} \le 2\textcolor{cornflowerblue}{C}\|u_0\|_{H^1},\  \|\sqrt{\mathcal{K}}_\lambda\|_{L^{q}(0,\infty; L^{r})}\le 2\textcolor{cornflowerblue}{C}\|u_0\|_{H^1}\right\},
\end{equation}
that we endow with the following distance
\begin{equation}
    \label{distance}
    d(u,v)= \|u-v\|_{L^{q}(0,\infty; L^{r})}+\|u-v\|_{L^{\theta}(0,\infty; L^{r})}.
\end{equation}
One can easily check that $(q,r)$ is admissible, $(\theta,r)$ is $s_c-$admissible and $(\tilde{\theta},r)$ is $(-s_c)-$admissible in the sense of Definition \ref{df}, where $s_c$ is defined by \eqref{s-c}. 

Arguing as in the proof of Theorem 1.8 in \cite{cg} and using Strichartz estimates together with Lemma \ref{lem-4.4}, the function $\Phi$ defines a contraction on the complete metric space $(\mathbf{X}, d)$ for $\varepsilon>0$ small enough.

The proof of Proposition \ref{prop3} is thus achieved via the Banach fixed-point theorem.
\end{proof}

Now, to conclude the proof of global existence in Theorem \ref{t1}, it remains to show that we can make $\left\|U_{\ba,\lambda}(\cdot)u_0\right\|_{L^{\theta}(0,\infty; L^{r})}$ small enough for large $\underline{\ba}$. To this end, we use the Sobolev embedding $H^1 \hookrightarrow L^r$ and \eqref{ais}, and we obtain
\begin{equation}
\label{Glo-epsi}
\begin{split}
\left\|U_{\ba,\lambda}(\cdot)u_0\right\|_{L^{\theta}(0,\infty; L^{r})}^\theta&=\int_0^\infty\,{\rm e}^{-\theta\baa(t)}\,\|{\rm e}^{-it\mathcal K_\lambda}u_0\|_{L^{r}}^{\theta}\,dt\\
&\lesssim \|u_0\|_{H^{1}}^{\theta}\,\int_0^\infty\, {\rm e}^{-\theta {\underline{\ba}} t}\,dt
\lesssim \frac{\|u_0\|_{H^{1}}^{\theta}}{\theta \,{\underline{\ba}}}.
\end{split}
\end{equation}
From \eqref{Glo-epsi} and Proposition \ref{prop3} we easily deduce that $T^*=\infty$ if $\underline{\ba}\gtrsim \|u_0\|_{H^{1}}^{\theta}$ ($\ba^*=C \|u_0\|_{H^{1}}^{\theta}$). 

This finishes the proof of small data global existence part of Theorem \ref{t1}. Furthermore, owing to the definition of the space $\mathbf{X}$, the global solution enjoys the bound  $\|u\|_{L^{\infty}(0,\infty; H^{1}_\lambda)} \le 2C_s\|u_0\|_{H^1}$, which yields the scattering part of Theorem \ref{t1}.
\section{Energy-critical regime}\label{crit}
This section focuses on proving Theorem \ref{LWP}. The core of our argument relies on a key auxiliary result, which we present in the following proposition.

\begin{prop}\label{LWP1}  
Let $N \geq 3$, $0 < b < \frac{4}{N}$, and $p =1+ \frac{4-2b}{N-2}$. Suppose that the parameter $\lambda$ satisfies  
$$  
\lambda > -\frac{(N+2-2b)^2-4}{(N+2-2b)^2}\lambda_N,  
$$  
where $\lambda_N$ is defined in \eqref{lambda_N}. For every initial data $u_0 \in H^1(\R^N)$, there exists  $\ba^* > 0$  such that if $\underline{\ba} \geq \ba^*$ (with $\underline{\ba}$ given by \eqref{ais}), the problem \eqref{DINLS-v} has a unique global solution $v$   such that
$$  
v \in C(\R_+; H^1(\R^N)) \cap L^{\alpha(r)}(\R_+; W^{1,r}(\R^N)),  
$$  
where $r = \frac{2N(N+2-2b)}{N^2-2Nb+4},$    
and the pair $(\alpha(r), r)$ satisfies the admissibility condition in \eqref{adm}. Furthermore, the solution satisfies the uniform bound  
\begin{equation}\label{est-local} 
\|v(t)\|_{H^1} \leq 2C\|u_0\|_{H^1}, \quad \text{for all \ } t \ge 0,  
\end{equation}
where $C > 0$ is a constant.  
\end{prop}  
\begin{proof}[Proof of Proposition \ref{LWP1}]
   We follow steps similar to those in the proof of \cite[Theorem 1.1]{Kim} with the necessary modifications.  For $M:=2 C \|u_0\|_{H^1}$, we define
    \begin{displaymath}
        \mathcal{X}:=\left\{v \in  L^{\alpha(r)}(\R_+; W^{1,r}): \|v\|_{L^{\alpha(r)}(\R_+;W^{1,r})} \le M \right\},
    \end{displaymath}
    and 
    \begin{equation}
    \label{Psi}
    \Psi(v)(t):=U_{\lambda}(t)u_0 + i \int_0^t\,e^{-(p-1)\baa(\tau)} U_{\lambda}(t-\tau)\left[|\cdot|^{-b}|v(\tau)|^{p-1}v(\tau)\right]d\tau.
\end{equation}

We emphasize that, by Lemma~\ref{Complete}, the function space $\mathcal{X}$ endowed with the metric
$$
d(v_1,v_2) := \|v_1 - v_2\|_{L^{\alpha(r)}(\mathbb{R}_+; L^r(\mathbb{R}^N))}
$$
forms a complete metric space. This completeness will be crucial for our subsequent fixed-point argument.\\

In the sequel, we will verify the conditions required in the application of the Banach fixed-point theorem. For that purpose, we show that $\Psi$ maps $\mathcal{X}$ to itself and it is a contraction for $\underline{\ba}$ large enough. 

Note that the function $\tau \mapsto e^{-(p-1)\baa(\tau)}$ is bounded for all $\tau \ge 0$. Hence, it will be ignored in the next estimates in the remainder  of this proof.

For $v \in \mathcal{X}$, by Strichartz estimate and Lemma \ref{NLE}, we obtain that

\begin{equation}\label{Contract}
\begin{split}
    \|\Psi(v)\|_{L^{\alpha(r)}(\R_+;W^{1,r})} &\le C \|u_0\|_{H^1} + C M^{p-1} \|e^{-(p-1)\baa(t)}v\|_{L^{\alpha(r)}(\R_+;W^{1,r})}\\
    &\le \frac{M}{2} + C M^{p-1}e^{-(p-1)\underline{\ba}} \|v\|_{L^{\alpha(r)}(\R_+;W^{1,r})}\\
     &\le \frac{M}{2} + C M^{p}e^{-(p-1)\underline{\ba}}.
    \end{split}
\end{equation}

If we choose $\ba^*>0$ such that $C M^{p-1}e^{-(p-1)\ba^*}\leq \frac{1}{2}$, then for  $\underline{\ba}\geq \ba^*$, we get by \eqref{Contract}, 

\begin{equation}\label{Contract2}
    \|\Psi(v)\|_{L^{\alpha(r)}(\R_+;W^{1,r})} \le \frac{M}{2} + \frac{M}{2}=M.
\end{equation}

Now, let $v_1, v_2 \in \mathcal{X}$. Again employing Strichartz estimate and Lemma \ref{NLE}, we end up with the following estimate for $\underline{\ba}\geq \ba^*$:

\begin{equation}\label{Contract1}
\begin{split}
    \|\Psi(v_1)-\Psi(v_2)\|_{L^{\alpha(r)}(\R_+;L^{r})} 
    &\le  C  e^{-(p-1)\underline{\ba}}\left\||x|^{-b}\left( |v_1|^{p-1}v_1-|v_2|^{p-1}v_2\right)\right\|_{L^{\alpha(\sigma)'}(\R_+;L^{\sigma'})}\\
    &\le  C e^{-(p-1)\underline{\ba}}\left( \|v_1\|^{p-1}_{L^{\alpha(r)}(\R_+;\dot{W}^{1,r})} +\|v_2\|^{p-1}_{L^{\alpha(r)}(\R_+;\dot{W}^{1,r})}\right) \|v_1-v_2\|_{L^{\alpha(r)}(\R_+;L^{r})}\\
    &\le  C e^{-(p-1)\underline{\ba}}M^{p-1} \|v_1-v_2\|_{L^{\alpha(r)}(\R_+;L^{r})}\\
    & \le \frac{1}{2}  \|v_1-v_2\|_{L^{\alpha(r)}(\R_+;L^{r})}.
    \end{split}
\end{equation}
This guarantees the existence of a unique fixed point $v \in \mathcal{X}$ for the mapping $\Psi$, that is, $\Psi(v) = v$.

The regularity $v \in C(\mathbb{R}_+; H^1)$ follows directly from the identity $\Psi(v) = v \in \mathcal{X}$ and the application of Lemma \ref{NLE}. To establish the estimate \eqref{est-local}, we combine the Strichartz estimates, the definition of the space $\mathcal{X}$, the relation $M = 2C\|u_0\|_{H^1}$, and arguments similar to those used in the proof of \eqref{Contract}.

This completes the proof of Proposition \ref{LWP1}.

\end{proof}

\begin{rem}
\rm
The damping term plays a crucial role here. Indeed, unlike the undamped case, which requires restricting the time interval to ensure contraction for local existence, our large damping assumption directly enables the fixed-point argument without shortening the existence time. This gives rise to global existence directly in the damped case.
\end{rem}
\begin{proof}[Proof of Theorem \ref{LWP}]
Let $v$ be the global solution of \eqref{DINLS-v} given by Proposition \ref{LWP1}.  The Duhamel formula for \eqref{DINLS-v} yields
\begin{equation}
    \label{Duh-v}
    v(t)=U_{\lambda}(t)v(t_0) + i \int_{t_0}^t\,e^{-(p-1)\baa(\tau)} U_{\lambda}(t-\tau)\left[|\cdot|^{-b}|v(\tau)|^{p-1}v(\tau)\right]d\tau,
\end{equation}
where $t_0>0$ is arbitrary chosen. Let $T>t_0$. Using Strichartz inequality, Lemma \ref{NLE}, \eqref{est-local} and \eqref{ais}, we infer that
\begin{equation}
    \label{boots-est-1}
  \|v\|_{L^{\alpha(r)}(t_0, T;\ W^{1,r})} \leq C\|u_0\|_{H^1}+C {\rm e}^{-(p-1)\underline{\ba}t_0}\,\|v\|_{L^{\alpha(r)}(t_0, T;\ W^{1,r})}^{p}.
\end{equation}
Applying Lemma \ref{boots}, we obtain that 
\begin{equation}
    \label{boots-est-2}
    \|v\|_{L^{\alpha(r)}(t_0, T;\ W^{1,r})} \lesssim \|u_0\|_{H^1}.
\end{equation}
Letting $T\to \infty$ in \eqref{boots-est-2}, we get
\begin{equation}
    \label{boots-est-3}
    \|v\|_{L^{\alpha(r)}(t_0, \infty;\ W^{1,r})} \lesssim \|u_0\|_{H^1}.
\end{equation}
From \eqref{boots-est-3}, the relation  $v(t,x)=e^{\baa(t)}u(t,x)$ and \eqref{ais}, we have
\begin{equation}
    \label{boots-est-4}
    \|u\|_{L^{\alpha(r)}(t_0, \infty;\ W^{1,r}(\R^N))} \lesssim {\rm e}^{-\underline{\ba}t_0}\,\|u_0\|_{H^1}.
\end{equation}
This yields the scattering for \eqref{DINLS}, and hence the proof of Theorem \ref{LWP} is complete.
\end{proof}

\section{Conclusion and open questions}\label{cr}
The present work offers a new avenue in the study of the damped INLS equation with an inverse-square potential, advancing our understanding of nonlinear dispersive equations influenced by a damping term. Global existence of solutions is established for various regimes, where the damping term plays a key role in preventing blow-up. Scattering is shown to occur under sufficiently large damping, highlighting the dissipation effect. The analysis combines refined Strichartz estimates, Hardy-type inequalities, and a  fixed-point argument. One of the key differences in this setting is that the standard ground state argument fails due to the presence of time-dependent damping. 

Despite these advances, several important questions remain open including, but not limited to,  optimizing the damping threshold, extending the results to non-radial data,  opening to other singular (repulsive) potentials and investigating generalizations to fractional operators.


\vspace{1cm}

\hrule 

\vspace{0.3cm}
\noindent{\bf\large Declarations.} {\em On behalf of all authors, the corresponding author states that there is no conflict of interest. No data-sets were generated or analyzed during the current study.}

	\vspace{0.7cm}

 \hrule 


\end{document}